\documentclass[a4paper,12pt]{article}
\usepackage{amsmath}
\usepackage{amsthm}
\usepackage{amssymb}
\usepackage{enumerate}
\usepackage{url}
\usepackage[utf8]{inputenc}
\usepackage{comment}

\title{Extensions of Bougerol's identity in law and the associated anticipative path transformations}

\author{Yuu Hariya\thanks{Supported in part by JSPS KAKENHI Grant Number 17K05288}}

\date{\empty}
\numberwithin{equation}{section}

\theoremstyle{plain}

\newtheorem{thm}{Theorem}[section]

\newtheorem{prop}{Proposition}[section]

\newtheorem{cor}{Corollary}[section]

\newtheorem{lem}{Lemma}[section]

\theoremstyle{definition}

\theoremstyle{remark}
\newtheorem{rem}{Remark}[section]

\begin{document}

\newcommand\ND{\newcommand}
\newcommand\RD{\renewcommand}

\ND\N{\mathbb{N}}
\ND\R{\mathbb{R}}
\ND\Q{\mathbb{Q}}
\ND\C{\mathbb{C}}

\ND\F{\mathcal{F}}

\ND\kp{\kappa}

\ND\ind{\boldsymbol{1}}

\ND\al{\alpha }
\ND\la{\lambda }
\ND\La{\Lambda }
\ND\ve{\varepsilon}
\ND\Om{\Omega}

\ND\ga{\gamma}

\ND\lref[1]{Lemma~\ref{#1}}
\ND\tref[1]{Theorem~\ref{#1}}
\ND\pref[1]{Proposition~\ref{#1}}
\ND\sref[1]{Section~\ref{#1}}
\ND\ssref[1]{Subsection~\ref{#1}}
\ND\aref[1]{Appendix~\ref{#1}}
\ND\rref[1]{Remark~\ref{#1}} 
\ND\cref[1]{Corollary~\ref{#1}}
\ND\eref[1]{Example~\ref{#1}}
\ND\fref[1]{Fig.\ {#1} }
\ND\lsref[1]{Lemmas~\ref{#1}}
\ND\tsref[1]{Theorems~\ref{#1}}
\ND\dref[1]{Definition~\ref{#1}}
\ND\psref[1]{Propositions~\ref{#1}}
\ND\rsref[1]{Remarks~\ref{#1}}
\ND\sssref[1]{Subsections~\ref{#1}}

\ND\pr{\mathbb{P}}
\ND\ex{\mathbb{E}}
\ND\br{W}

\ND\prb[2]{P^{(#1)}_{#2}}
\ND\exb[2]{E^{(#1)}_{#2}}

\ND\eb[1]{e^{B_{#1}}}
\ND\ebm[1]{e^{-B_{#1}}}
\ND\hbe{\Hat{\beta}}
\ND\hB{\Hat{B}}
\ND\argsh{\mathrm{Argsh}\,}
\ND\zmu{z_{\mu}}
\ND\GIG[3]{\mathrm{GIG}(#1;#2,#3)}
\ND\gig[3]{I^{(#1)}_{#2,#3}}
\ND\argch{\mathrm{Argch}\,}

\ND\vp{\varphi}
\ND\eqd{\stackrel{(d)}{=}}
\ND\db[1]{B^{(#1)}}
\ND\da[1]{A^{(#1)}}
\ND\dz[1]{Z^{(#1)}}
\ND\Z{\mathcal{Z}}

\ND\anu{\alpha ^{(\nu )}}

\ND\Ga{\Gamma}
\ND\calE{\mathcal{E}}
\ND\calD{\mathcal{D}}

\ND\f{F}
\ND\g{G}

\ND\tr{\mathbb{T}}
\ND\T{T}

\ND{\rmid}[1]{\mathrel{}\middle#1\mathrel{}}

\def\thefootnote{{}}

\maketitle 
\begin{abstract}
Let $B=\{ B_{t}\} _{t\ge 0}$ be a one-dimensional standard 
Brownian motion and denote by $A_{t},\,t\ge 0$, the quadratic 
variation of the geometric Brownian motion $e^{B_{t}},\,t\ge 0$. 
Bougerol's celebrated identity (1983) asserts that, if 
$\beta =\{ \beta (t)\} _{t\ge 0}$ is another Brownian motion 
independent of $B$, then $\beta (A_{t})$ is identical in law with 
$\sinh B_{t}$ for every fixed $t>0$. In this paper, we extend 
Bougerol's identity to an identity in law for processes up to time 
$t$, which exhibits a certain invariance of the law of Brownian motion. 
The extension is described in terms of anticipative transforms 
of $B$ involving $A_{t}$ as an anticipating factor. A Girsanov-type 
formula for those transforms is shown. An extension of a variant 
of Bougerol's identity is also presented.
\footnote{Mathematical Institute, Tohoku University, Aoba-ku, Sendai 980-8578, Japan}
\footnote{E-mail: hariya@tohoku.ac.jp}
\footnote{{\itshape Keywords and Phrases}:~{Brownian motion}; {exponential functional}; {Bougerol's identity}; {anticipative path transformation}}
\footnote{{\itshape MSC 2020 Subject Classifications}:~Primary~{60J65}; Secondary~{60J55}, {60G30}}
\end{abstract}

\section{Introduction}\label{;intro}

Let $B=\{ B_{t}\} _{t\ge 0}$ be a one-dimensional standard 
Brownian motion. For every $\mu \in \R $, we denote by 
$\db{\mu }=\{ \db{\mu }_{t}:=B_{t}+\mu t\} _{t\ge 0}$ the 
Brownian motion with drift $\mu $, to which we associate 
the exponential additive functional $\da{\mu }_{t},\,t\ge 0$, 
defined by 
\begin{align*}
 \da{\mu }_{t}:=\int _{0}^{t}e^{2\db{\mu }_{s}}ds;  
\end{align*}
when $\mu =0$, we simply write $A_{t}$ for $\da{\mu }_{t}$. 
This additive functional is the quadratic variation of the 
geometric Brownian motion $e^{\db{\mu }_{t}},\,t\ge 0$, and 
appears in a number of fields such as mathematical finance, 
diffusion processes in random environments, stochastic 
analysis of Laplacians on hyperbolic spaces, and so on; see 
the detailed surveys \cite{mySI,mySII} by Matsumoto and Yor. 

Although the law of $A_{t}$, as well as of $\da{\mu }_{t}$, has 
a rather complicated expression (refer to the beginning of 
the \ssref{;sspcmain2} in this respect), there are alternative 
ways to understand it, 
one of which is provided by Bougerol's celebrated identity 
(\cite{bou}) asserting that, if $\beta =\{ \beta (t)\} _{t\ge 0}$ 
denotes another one-dimensional standard Brownian motion and 
if it is independent of $B$, then for every fixed $t>0$, 
\begin{align}
 \beta (A_{t})&\eqd \sinh B_{t}, \label{;boug}
\intertext{or, more generally, for every fixed $t>0$ and $x\in \R $,}
 \eb{t}\sinh x+\beta (A_{t})&\eqd \sinh (x+B_{t}); \label{;gboug}
\end{align}
the latter is originally due to Alili--Gruet \cite[Proposition~4]{ag}. 
For a straightforward derivation of \eqref{;boug} 
by means of stochastic differential equations, which works 
for \eqref{;gboug} as well, 
we refer the reader to an inventive argument 
\cite[Appendix]{cmy} by Alili and Dufresne.
From the former identity \eqref{;boug} in particular, we easily 
obtain the moment formula 
\begin{align*}
 \ex \!\left[ 
 \left( A_{t}\right) ^{\nu }
 \right] 
 =\frac{\sqrt{\pi }}{2^{\nu }\Gamma (\nu +1/2)}
 \ex \!\left[ |\sinh B_{t}|^{2\nu }\right] ,\quad \nu >-1/2,
\end{align*}
where $\Gamma $ is the gamma function. 
The expression of the right-hand side is fairly tractable. 
For a detailed account of Bougerol's identity and 
recent progress in its study such as extensions to other 
processes, see the survey \cite{vak} by Vakeroudis.

Define $C([0,\infty );\R )$ to be the space of 
continuous functions $\phi :[0,\infty )\to \R $.
We set 
\begin{align}\label{;defa}
 A_{t}(\phi ):=\int _{0}^{t}e^{2\phi _{s}}\,ds,\quad t\ge 0,
\end{align}
so that $\da{\mu }_{t}=A_{t}(\db{\mu })$ and 
$A_{t}=A_{t}(B)$, $t\ge 0$. In this paper, we show that 
Bougerol's identity may be extended to an identity in law 
for processes. Unless otherwise stated, we fix $t>0$ and 
for every $z\in \R $, consider the path transformation 
$\tr _{z}$ defined by 
\begin{align}\label{;ttrans}
 \tr _{z}(\phi )(s):=\phi _{s}-\log \left\{ 
 1+\frac{A_{s}(\phi )}{A_{t}(\phi )}\left( e^{z}-1\right) 
 \right\} ,\quad 0\le s\le t, 
\end{align}
which maps the space $C([0,t];\R )$ of real-valued 
continuous functions over $[0,t]$ to itself. We also denote by 
\begin{align*}
 \argsh x\equiv \log \left( x+\sqrt{1+x^{2}}\right) ,\quad x\in \R , 
\end{align*}
the inverse function of the hyperbolic sine function. One of the 
main results of the paper is stated as follows: 

\begin{thm}\label{;tmain1}
Suppose $B$ and $\beta $ are independent. 
Then, for every fixed $x\in \R $, the process 
\begin{align}\label{;ttrans1}
 \tr _{x+B_{t}-\argsh (e^{B_{t}}\sinh x +\beta (A_{t}))
 }
 (B)(s),\quad 0\le s\le t, 
\end{align}
is identical in law with $\{ B_{s}\} _{0\le s\le t}$.
\end{thm}

To see that the above theorem extends Bougerol's identity, we 
evaluate \eqref{;ttrans1} at $s=t$. Then the theorem entails 
\begin{align*}
 -x+\argsh \bigl( e^{B_{t}}\sinh x +\beta (A_{t})\bigr) 
 \eqd B_{t},
\end{align*}
which is nothing but \eqref{;gboug}. 

\begin{rem}\label{;rnotable}
A notable thing about the theorem is 
that, in view of \eqref{;gboug},  
the difference of the two random variables with same 
law, $x+B_{t}$ and $\argsh \bigl( e^{B_{t}}\sinh x +\beta (A_{t})\bigr) $, 
is substituted into $z$ of the transform $\tr _{z}(B)$, and 
the resulting process is distributed as 
$\tr _{0}(B)$ as if those two random variables were identical.
\end{rem}
In \sref{;sptmain1}, we prove \tref{;tmain1} in an extended form. 
See \tref{;tmain1d}. 

Following the notation used in a series of 
papers \cite{myPI,myPII,my2003,mySII} by Matsumoto and Yor, 
we set 
\begin{align*}
 \dz{\mu }_{t}:=e^{-\db{\mu }_{t}}\!\da{\mu }_{t},\quad t\ge 0, 
\end{align*}
to which we associate the path transformation $Z$ defined by 
\begin{align}\label{;defz}
 Z_{t}(\phi ):=e^{-\phi _{t}}A_{t}(\phi ),\quad t\ge 0,
\end{align}
for $\phi \in C([0,\infty );\R )$, so that 
$\dz{\mu }_{t}=Z_{t}(\db{\mu }),\,t\ge 0$. 
With slight abuse of notation, we simply write $Z_{t}$ for 
$Z_{t}(B)$ as well. In this paper, we also prove an absolutely 
continuous relationship for the anticipative transforms 
$\tr _{z}(B),\,z\in \R $, of the Brownian motion $B$, and reveal 
how \tref{;tmain1} is connected to it.

\begin{thm}\label{;tmain2}
Fix $z\in \R $. For every nonnegative measurable functional 
$F$ on $C([0,t];\R )$, we have 
\begin{align}\label{;t21}
 \ex \!\left[ 
 F\bigl( \tr _{z}(B)(s),s\le t\bigr) 
 \right] 
 =\ex \!\left[ 
 \exp \left\{ 
 \frac{\cosh B_{t}-\cosh (z+B_{t})}{Z_{t}}
 \right\} F(B_{s},s\le t)
 \right] ,
\end{align}
or, equivalently, 
\begin{align}\label{;t22}
 \ex \!\left[ 
 F(B_{s},s\le t)
 \right] 
 =\ex \!\left[ 
 \exp \left\{ 
 \frac{\cosh B_{t}-\cosh (z+B_{t})}{Z_{t}}
 \right\} 
 F\bigl( \tr _{-z}(B)(s),s\le t\bigr)
 \right] .
\end{align}
\end{thm}

A link between the above theorem and the Malliavin calculus 
will be mentioned at the end of the paper.

For each $\phi \in C([0,\infty );\R )$ such that 
$A_{\infty }(\phi ):=\lim \limits_{t\to \infty }A_{t}(\phi )<\infty $, 
we define 
\begin{align}\label{;ttransd}
 \tr ^{*}_{z}(\phi )(s):=\phi _{s}-\log \left\{ 
 1+\frac{A_{s}(\phi )}{A_{\infty }(\phi )}\left( e^{z}-1\right) 
 \right\} ,\quad s\ge 0.
\end{align}
By applying the Cameron--Martin relation in \tref{;tmain2} and 
passing to the limit as $t\to \infty $, we immediately obtain 

\begin{cor}\label{;cmain2}
Let $\mu >0$ and fix $z\in \R $. For every nonnegative 
measurable functional $F$ on $C([0,\infty );\R )$, we have 
\begin{align}\label{;ec1}
 \ex \!\left[ 
 F\bigl( \tr ^{*}_{z}(\db{-\mu })(s),s\ge 0\bigr) 
 \right] 
 =e^{-\mu z}\ex \!\left[ 
 \exp \left( \frac{1-e^{-z}}{2\da{-\mu }_{\infty }}\right) 
 F\bigl( \db{-\mu }_{s},s\ge 0\bigr) 
 \right] ,
\end{align}
or, equivalently, 
\begin{align}\label{;ec2}
 \ex \!\left[ 
 F\bigl( \db{-\mu }_{s},s\ge 0\bigr) 
 \right] 
 =e^{-\mu z}\ex \!\left[ 
 \exp \left( \frac{1-e^{-z}}{2\da{-\mu }_{\infty }}\right) 
 F\bigl( \tr ^{*}_{-z}(\db{-\mu })(s),s\ge 0\bigr) 
 \right] .
\end{align}
\end{cor}

When $\mu >0$, we know 
$\da{-\mu }_{\infty }=\lim \limits_{t\to \infty }\da{-\mu }_{t}<\infty $ 
a.s.; in fact, it is shown by Dufresne 
\cite[Proposition~4.4.4(b)]{duf} that 
\begin{align}\label{;duf}
 \da{-\mu }_{\infty }
 &\eqd \frac{1}{2\ga _{\mu }}, 
\end{align}
where $\ga _{\mu }$ is a gamma random variable with parameter 
$\mu $: 
\begin{align*}
 \pr (\ga _{\mu }\in du)
 &=\frac{1}{\Gamma (\mu )}u^{\mu -1}e^{-u}\,du,\quad u>0. 
\end{align*}

\begin{rem}\label{;rduf}
Taking $F\equiv 1$ and $z=-\log (1+\al )$ for $\al >-1$, 
we may deduce Dufresne's identity \eqref{;duf} 
from \cref{;cmain2} in such a way that, for any $\al >-1$,  
\begin{align*}
 \ex \!\left[ 
 \exp \left( 
 -\frac{\al }{2\da{-\mu }_{\infty }}
 \right) 
 \right] &=\left( \frac{1}{1+\al }\right) ^{\mu }\\
 &=\frac{1}{\Gamma (\mu )}\int _{0}^{\infty }du\,
 u^{\mu -1}e^{-(1+\al )u}, 
\end{align*}
which implies \eqref{;duf} thanks to the injectivity of the 
Laplace transform.
\end{rem}

An intriguing fact which may be deduced from \cref{;cmain2} 
is the following invariance of the law of $\db{-\mu}$:
\begin{align}\label{;inv} 
 \Bigl\{ 
 \tr ^{*}_{\log ( 2\ga _{\mu }\da{-\mu }_{\infty }) }
 (\db{-\mu })(s) 
 \Bigr\} _{s\ge 0}
 \eqd \bigl\{ \db{-\mu }_{s}\bigr\} _{s\ge 0}.
\end{align}
See \pref{;pintr}. Here $B$ and $\ga _{\mu }$ are 
independent on the left-hand side. Noting 
$
\log \bigl( 2\ga _{\mu }\da{-\mu }_{\infty }\bigr) 
=\log (2\ga _{\mu })-\log \bigl( 1/\da{-\mu }_{\infty }\bigr) 
$ 
and remembering Dufresne's identity 
\eqref{;duf}, we may compare the above invariance with 
\rref{;rnotable}. Our research is also inspired by 
Donati-Martin--Matsumoto--Yor \cite{dmy}, in which 
a family of path transformations on $C([0,\infty );\R )$, denoted by 
$T_{\al },\,\al >0$, is introduced and its properties are 
investigated. We discuss its connection with 
transformations \eqref{;ttransd} and \cref{;cmain2}, 
as well as with the above-mentioned invariance; see 
\ssref{;sspcmain2}.

Let $\ve $ denote a Rademacher (or symmetric Bernoulli) 
random variable taking values $\pm 1$ with equal probability. 
The following variant of Bougerol's identity 
is known (see, e.g., \cite[Corollary~2.2]{vak}): for every 
fixed $t>0$, 
\begin{align}\label{;boug1}
 \beta \bigl( \da{1}_{t}\bigr) \eqd \sinh \db{\ve }_{t}, 
\end{align}
where, on the left-hand side, $\beta =\{ \beta (s)\} _{s\ge 0}$ 
is a one-dimensional standard Brownian motion independent of 
$B$ as before, and on the right-hand side, $\ve $ is independent 
of $B$. \tref{;tmain1} (or, more precisely, \tref{;tmain1d}) enables 
us to obtain an extension of \eqref{;boug1}. In the sequel, 
given a real-valued process $X=\{ X_{s}\} _{s\ge 0}$ 
and a point $a\in \R $, we denote by $\tau _{a}(X)$ the first 
hitting time of $X$ to the level $a$: 
\begin{align*}
 \tau _{a}(X):=\inf \{ s\ge 0;\,X_{s}=a\}  
\end{align*}
with the convention $\inf \emptyset =\infty $. 

\begin{thm}\label{;tmain3}
Let $\beta =\{ \beta (s)\} _{s\ge 0}$ and 
$\hB =\{ \hB _{s}\} _{s\ge 0}$ be one-dimensional standard 
Brownian motions. The following identity in law holds: 
\begin{equation}\label{;etmain3}
\begin{split}
 &\left( 
 \left\{ 
 \tr _{\db{1}_{t}-\argsh \beta (\da{1}_{t})}(\db{1})(s)
 \right\} _{0\le s\le t},\,e^{-2\db{1}_{t}}\!\da{1}_{t}
 \right) \\
 &\eqd \left( 
 \{ \db{\ve }_{s}\} _{0\le s\le t},\,
 \tau _{1}(\hB ^{(\cosh \db{\ve }_{t}/\dz{\ve }_{t})})
 \right) ,
\end{split}
\end{equation}
where, on the left-hand side, $B$ and $\beta $ are independent, 
and on the right-hand side, $B$, $\hB $ and $\ve $ are independent. 
In particular, we have  
\begin{align*}
 \left( 
 \beta \bigl( \da{1}_{t}\bigr) ,\,e^{-2\db{1}_{t}}\!\da{1}_{t},\,
 \dz{1}_{t}
 \right) \eqd 
 \left( 
 \sinh \db{\ve }_{t},\,\tau _{1}(\hB ^{(\cosh \db{\ve }_{t}/\dz{\ve }_{t})}),
 \,\dz{\ve }_{t}
 \right) .
\end{align*}
\end{thm}

The last identity may be regarded as a counterpart to 
\cite[Theorem~1.2]{har} with $x=0$ and $\tau =t$ therein, 
in the presence of the drift $\mu =1$; in addition, the identity 
between the third coordinates is consistent with the fact 
\cite[Theorem~1.6(ii)]{myPI} that, for each $\mu \in \R $, 
\begin{align}\label{;oppz}
 \{ \dz{\mu }_{t}\} _{t\ge 0}
 \eqd \{ \dz{-\mu }_{t}\} _{t\ge 0}.
\end{align}

The rest of the paper is organized as follows. In \sref{;sptrans}, we 
summarize some properties, such as a semigroup property, 
of the path transformations $\tr _{z}$, which will be referred 
to throughout the paper. 
\tref{;tmain1} is proven in \sref{;sptmain1} by extending it 
to \tref{;tmain1d}; we state and prove 
\tref{;tmain1d} in \ssref{;ssptmain1d}, by preparing \lref{;lext} 
whose proof is given in \ssref{;ssplext}. In \sref{;sptmain2}, we 
prove \tref{;tmain2} and its \cref{;cmain2} after discussing how 
\tsref{;tmain1} and \ref{;tmain1d} are connected to \tref{;tmain2}; 
the proof of \tref{;tmain2} is given in \ssref{;ssptmain2} while 
\cref{;cmain2} is proven in \ssref{;sspcmain2}. 
In \sref{;sptmain3}, we give a proof of \tref{;tmain3}. 
The paper is concluded with \sref{;scr}, in which two remarks are 
given: one is about a possible extension of \tsref{;tmain1}, 
\ref{;tmain1d} and \ref{;tmain2} to the situation where the terminal 
time $t$ therein is replaced by any positive and finite stopping time of 
the process $\{ Z_{s}\} _{s\ge 0}$; and the other a plausible 
derivation of \tref{;tmain2} in the framework of the Malliavin 
calculus.
\smallskip 

In the sequel, when we say a Brownian motion, it always 
refers to a one-dimensional and standard one. We equip 
the spaces $C([0,t];\R )$ and $C([0,\infty );\R )$ with 
topologies of uniform convergence and local uniform 
convergence, respectively; 
in particular, functionals on these spaces are said to be 
measurable if they are Borel-measurable with respect to 
those topologies.

\section{Properties of $\tr _{z}$}\label{;sptrans}

In this section, we investigate properties of the 
transformations $\tr _{z},\,z\in \R $, defined by 
\eqref{;ttrans} with $t>0$ fixed. 
Accordingly, the two transformations $A$ and $Z$ defined by 
\eqref{;defa} and \eqref{;defz}, respectively, are restricted on 
$C([0,t];\R )$. These are related via 
\begin{align}\label{;deria}
 \frac{d}{ds}\frac{1}{A_{s}(\phi )}
 =-\left\{ \frac{1}{Z_{s}(\phi )}\right\} ^{2},\quad 0<s\le t, 
\end{align}
for any $\phi \in C([0,t];\R )$.
We also consider the time reversal which we denote by $R$: 
\begin{align}\label{;trev}
 R(\phi )(s):=\phi _{t-s}-\phi _{t},\quad 0\le s\le t,\,
 \phi \in C([0,t];\R ).
\end{align}

\begin{prop}\label{;pttrans}
We have the following properties \thetag{i}--\thetag{v}. 
\begin{itemize}
\item[\thetag{i}] For every $z\in \R $ and $\phi \in C([0,t];\R )$,
$\tr _{z}(\phi )(t)=\phi _{t}-z$.

\item[\thetag{ii}] For every $z\in \R $ and $\phi \in C([0,t];\R )$, 
\begin{align*}
 \frac{1}{A_{s}(\tr _{z}(\phi ))}=\frac{1}{A_{s}(\phi )}
 +\frac{e^{z}-1}{A_{t}(\phi )},\quad 0<s\le t.
\end{align*}
In particular, $A_{t}(\tr _{z}(\phi ))=e^{-z}A_{t}(\phi )$.

\item[\thetag{iii}] (Semigroup property) 
$\tr _{z}\circ \tr _{z'}=\tr _{z+z'}$ for any $z,z'\in \R $. 
In particular, 
\begin{align*}
 \tr _{z}\circ \tr _{-z}=\tr _{0}=I \quad \text{for any $z\in \R $}.
\end{align*}
Here $I:C([0,t];\R )\to C([0,t];\R )$ is the identity map.

\item[\thetag{iv}] $Z\circ \tr _{z}=Z$ for any $z\in \R $.

\item[\thetag{v}] For every $z\in \R $, 
$\tr _{z}\circ R\circ \tr _{z}=R$, and hence 
\begin{align*}
 R\circ \tr _{z}=\tr _{-z}\circ R.
\end{align*}

\end{itemize}
\end{prop}

\begin{proof}
\thetag{i} This follows immediately from the definition \eqref{;ttrans} 
of $\tr _{z}$.

\thetag{ii} This follows by a direct computation: for $z\neq 0$, we have 
\begin{align*}
 A_{s}(\tr _{z}(\phi ))
 &=\int _{0}^{s}du\,\frac{e^{2\phi _{u}}}{\left\{ 
 1+\frac{A_{u}(\phi )}{A_{t}(\phi )}(e^{z}-1)
 \right\} ^{2}}\\
 &=\frac{A_{t}(\phi )}{e^{z}-1}\left\{ 
 1-\frac{1}{1+\frac{A_{s}(\phi )}{A_{t}(\phi )}(e^{z}-1)}
 \right\} \\
 &=\frac{A_{s}(\phi )}{1+\frac{A_{s}(\phi )}{A_{t}(\phi )}(e^{z}-1)}, 
\end{align*}
which entails the claim. The case $z=0$ is obvious since 
$\tr _{0}(\phi )=\phi $.

\thetag{iii} It suffices to show that, for each 
$\phi \in C([0,t];\R )$, 
\begin{align}\label{;epp1}
 A_{s}\bigl( (\tr _{z}\circ \tr _{z'})(\phi )\bigr) 
 =A_{s}(\tr _{z+z'}(\phi )),\quad 0\le s\le t; 
\end{align}
indeed, once this identity is shown, then taking the 
derivative with respect to $s$ on both sides leads to the 
conclusion. To this end, for every $0<s\le t$, repeated use of 
\thetag{ii} yields 
\begin{align*}
 \frac{1}{A_{s}\bigl( (\tr _{z}\circ \tr _{z'})(\phi )\bigr)}
 &=\frac{1}{A_{s}(\tr _{z}(\phi ))}
 +\frac{e^{z'}-1}{A_{t}(\tr _{z}(\phi ))}\\
 &=\frac{1}{A_{s}(\phi )}
 +\frac{e^{z}-1}{A_{t}(\phi )}
 +\frac{e^{z'}-1}{e^{-z}A_{t}(\phi )}\\
 &=\frac{1}{A_{s}(\tr _{z+z'}(\phi ))}, 
\end{align*}
which shows \eqref{;epp1}.

\thetag{iv} Noting \eqref{;deria} and taking the derivative 
with respect to $s$ on both sides of the displayed identity in 
property \thetag{ii}, we arrive at the conclusion.

\thetag{v} Since the case $z=0$ is obvious, we let $z\neq 0$. 
Pick $\phi \in C([0,t];\R )$ arbitrarily and set 
\begin{align*}
 \psi _{s}=(R\circ \tr _{z})(\phi )(s),\quad 0\le s\le t.
\end{align*}
A direct computation shows that 
\begin{align*}
 A_{s}(\psi )=e^{-2\phi _{t}+2z}\frac{A_{t}(\phi )}{e^{z}-1}
 \left\{ 
 \frac{1}{1+\frac{A_{t-s}(\phi )}{A_{t}(\phi )}(e^{z}-1)}-e^{-z}
 \right\} ,
\end{align*}
and in particular, $A_{t}(\psi )=e^{-2\phi _{t}+z}A_{t}(\phi )$, 
from which we have 
\begin{align*}
 1+\frac{A_{s}(\psi )}{A_{t}(\psi )}(e^{z}-1)
 =\frac{e^{z}}{1+\frac{A_{t-s}(\phi )}{A_{t}(\phi )}(e^{z}-1)}.
\end{align*}
Therefore, by the definition of $\tr _{z}$, 
\begin{align*}
 \tr _{z}(\psi )(s)&=\psi _{s}-\log \left\{ 
 1+\frac{A_{s}(\psi )}{A_{t}(\psi )}(e^{z}-1)
 \right\} \\
 &=\phi _{t-s}-\log \left\{ 
 1+\frac{A_{t-s}(\phi )}{A_{t}(\phi )}(e^{z}-1)
 \right\} 
 -\phi _{t}+z\\
 &\quad -\log \left\{ 
 \frac{e^{z}}{1+\frac{A_{t-s}(\phi )}{A_{t}(\phi )}(e^{z}-1)}
 \right\} \\
 &=\phi _{t-s}-\phi _{t}, 
\end{align*}
which proves 
the former assertion. The latter follows from the former and 
\thetag{iii}.
\end{proof}

We end this section with a lemma concerning the time 
reversal $R$ for later use.

\begin{lem}\label{;ltrev}
It holds that, for all $\phi \in C([0,t];\R )$, 
\begin{align*}
 Z_{t}(R(\phi )) =Z_{t}(\phi ).
\end{align*}
\end{lem}

\begin{proof}
By the definition of $Z$, 
\begin{align*}
 Z_{t}(R(\phi ))&=
 e^{-R(\phi )(t)}\int _{0}^{t}e^{2R(\phi )(s)}\,ds\\
 &=e^{\phi _{t}}\int _{0}^{t}e^{2(\phi _{t-s}-\phi _{t})}\,ds\\
 &=e^{-\phi _{t}}\int _{0}^{t}e^{2\phi _{s}}\,ds, 
\end{align*}
which shows the claim.
\end{proof}

\section{Proof of \tref{;tmain1}}\label{;sptmain1}

In this section, we give a proof of \tref{;tmain1}; we do this by 
providing an extension of the theorem. We keep $t>0$ fixed, and 
$\tr _{z},\,z\in \R $, are the 
transformations \eqref{;ttrans} acting on $C([0,t];\R )$.

\subsection{Extension of \tref{;tmain1} and its proof}\label{;ssptmain1d}

\tref{;tmain1} is part of the statement of the following theorem: 
\begin{thm}\label{;tmain1d}
Let $\beta =\{ \beta (s)\} _{s\ge 0}$ and 
$\hB=\{ \hB _{s}\} _{s\ge 0}$are Brownian motions 
independent of $B$. For each fixed $x\in \R $, we denote 
\begin{align*}
 \zeta =\argsh \bigl( e^{B_{t}}\sinh x +\beta (A_{t})\bigr) -(x+B_{t}), 
 && \T =\tau _{\cosh (x+B_{t})}(\hB ^{(\cosh x/Z_{t})}),
\end{align*}
for simplicity. Then the following identities in law hold: 
\begin{align}
 \left( \{ B_{s}\} _{0\le s\le t},\,\zeta \right) 
 &\eqd 
 \left( 
 \left\{ 
 \tr _{\log (A_{t}/\T )}(B)(s)\right\} _{0\le s\le t},\,\log (A_{t}/\T )
 \right) ,\label{;id1}\\
 \left( 
 \left\{ \tr _{-\zeta }(B)(s)\right\} _{0\le s\le t},\,A_{t}
 \right) 
 &\eqd 
 \left( \{ B_{s}\} _{0\le s\le t},\,T\right) .\label{;id1d}
\end{align}
In particular, both  
\begin{align*}
 \left\{ \tr _{\log (A_{t}/\T )}(B)(s)\right\} _{0\le s\le t} \quad 
 \text{and} \quad \left\{ \tr _{-\zeta }(B)(s)\right\} _{0\le s\le t}
\end{align*}
are Brownian motions.
\end{thm}

It is our finding (\cite[Theorem~1.2]{har}) that, for every fixed 
$x\in \R $,  
\begin{equation}\label{;ext0}
  \begin{split}
  &\left( \eb{t}\!\sinh x+\beta (A_{t}),\,A_{t},
  \,Z_{t}\right) \\
  &\eqd 
  \left( 
  \sinh (x+B_{t}),\,
  \tau _{\cosh (x+B_{t})}(\hB ^{(\cosh x/Z_{t})}),
  \,Z_{t}
  \right) .
  \end{split}
\end{equation}
The starting point of our proof of \tref{;tmain1d} is to 
observe that identity \eqref{;ext0} may be extended in such a 
way that 
\begin{equation}\label{;ext}
  \begin{split}
  &\left( \eb{t}\!\sinh x+\beta (A_{t}),\,A_{t},
  \,\{ Z_{s}\} _{0\le s\le t}\right) \\
  &\eqd 
  \left( 
  \sinh (x+B_{t}),\,
  \tau _{\cosh (x+B_{t})}(\hB ^{(\cosh x/Z_{t})}),
  \,\{ Z_{s}\} _{0\le s\le t}
  \right) .
  \end{split}
\end{equation}

\begin{lem}\label{;lext}
Identity \eqref{;ext} holds.
\end{lem}

We postpone the proof of \lref{;lext} to \ssref{;ssplext}. By using 
this lemma, the proof of \tref{;tmain1d} proceeds in the 
following way: 
\begin{proof}[Proof of \tref{;tmain1d}]
By noting \eqref{;deria} and by \lref{;lext}, 
\begin{align*}
 \left( 
 e^{B_{t}}\sinh x +\beta (A_{t}),\,
 A_{t},\,\left\{ 
 \frac{1}{A_{s}}-\frac{1}{A_{t}}
 \right\} _{0<s\le t}
 \right) \eqd 
 \left( 
 \sinh (x+B_{t}),\,\T ,\,\left\{ 
 \frac{1}{A_{s}}-\frac{1}{A_{t}}
 \right\} _{0<s\le t}
 \right) , 
\end{align*}
from which it follows that 
\begin{align*}
 \left( 
 \zeta +B_{t},
 \,\left\{ 
 \frac{1}{A_{s}}
 \right\} _{0<s\le t}
 \right) 
 \eqd 
 \left( 
 B_{t},\,\left\{ 
 \frac{1}{A_{s}}-\frac{1}{A_{t}}+\frac{1}{\T }
 \right\} _{0<s\le t}
 \right) .
\end{align*}
Note that, on the right-hand side, we have for every 
$0<s\le t$,  
\begin{equation}\label{;noted}
\begin{split}
 \frac{1}{A_{s}}-\frac{1}{A_{t}}+\frac{1}{\T }
 &=\frac{1}{A_{s}}+\frac{1}{A_{t}}\left( \frac{A_{t}}{\T }-1\right) \\
 &=\frac{1}{A_{s}\bigl( \tr _{\log (A_{t}/\T )}(B)\bigr) }, 
\end{split}
\end{equation}
where the second line follows from \pref{;pttrans}(ii). 
Therefore the last identity in law entails that 
\begin{align*}
 \left( 
 \zeta +B_{t},
 \,\{ A_{s}\} _{0\le s\le t}
 \right) 
 \eqd 
 \left( 
 B_{t},\,\left\{ 
 A_{s}\bigl( \tr _{\log (A_{t}/\T )}(B)\bigr) 
 \right\} _{0\le s\le t}
 \right) ,
\end{align*}
and hence, by differentiating the second coordinate 
with respect to $s$ on both sides, 
\begin{align*}
 \left( 
 \zeta +B_{t},
 \,\{ B_{s}\} _{0\le s\le t}
 \right) 
 \eqd 
 \left( 
 B_{t},\,\left\{ 
 \tr _{\log (A_{t}/\T )}(B)(s)
 \right\} _{0\le s\le t}
 \right) .
\end{align*}
This proves \eqref{;id1} because  
$B_{t}-\tr _{\log (A_{t}/\T )}(B)(t)=\log (A_{t}/\T )$ by 
\pref{;pttrans}\thetag{i}. 
To prove \eqref{;id1d}, observe from \eqref{;id1} the identity 
\begin{align}\label{;id2}
 \left( 
 \left\{ \tr _{-\zeta }(B)(s)\right\} _{0\le s\le t},\,\zeta 
 \right) 
 &\eqd 
 \left( \{ B_{s}\} _{0\le s\le t},\,\log (A_{t}/\T )\right) ,
\end{align}
which may be seen as a consequence of \pref{;pttrans}(iii) 
(see also \rref{;rid2}). Or, more convincingly, 
by \eqref{;id1} and as noted in \eqref{;noted},  
\begin{align*}
 \left( \left\{ \frac{1}{A_{s}}\right\} _{0<s\le t},\,\frac{1}{A_{t}},
 \,\zeta \right) 
 \eqd 
 \left( 
 \left\{ 
 \frac{1}{A_{s}}-\frac{1}{A_{t}}+\frac{1}{\T }
 \right\} _{0<s\le t},\,\frac{1}{\T },\,\log \frac{A_{t}}{\T }
 \right) ,
\end{align*}
from which we have, by \pref{;pttrans}\thetag{ii},  
\begin{align*}
 \left( 
 \left\{ \frac{1}{A_{s}\bigl( \tr _{-\zeta }(B)\bigr) }\right\} _{0<s\le t},
 \,\zeta 
 \right) 
 &=\left( 
 \left\{ \frac{1}{A_{s}}+\frac{e^{-\zeta }-1}{A_{t}}\right\} _{0<s\le t},
 \,\zeta \right) \\
 &\eqd \left( 
 \left\{ 
 \frac{1}{A_{s}}
 \right\} _{0<s\le t},\,\log \frac{A_{t}}{\T }
 \right) .
\end{align*}
Then the same argument as used in the 
first half of the proof leads to \eqref{;id2}. It now 
follows from \eqref{;id2} that, by \pref{;pttrans}\thetag{i}, 
\begin{align*}
 \left( 
 \left\{ \tr _{-\zeta }(B)(s)\right\} _{0\le s\le t},\,B_{t}+\zeta ,
 \,e^{-\zeta }
 \right) 
 &\eqd 
 \left( \{ B_{s}\} _{0\le s\le t},\,B_{t},
 \,T/A_{t}\right) ,
\end{align*}
and hence, by the definition $Z_{t}=e^{-B_{t}}A_{t}$ of $Z_{t}$,
\begin{align*}
 \left( 
 \left\{ \tr _{-\zeta }(B)(s)\right\} _{0\le s\le t},\,e^{B_{t}}
 \right) 
 &\eqd 
 \left( \{ B_{s}\} _{0\le s\le t},\,T/Z_{t}\right) .
\end{align*}
We multiply the second coordinates on the respective sides 
by $Z_{t}\bigl( \tr _{-\zeta }(B)\bigr) $ and $Z_{t}$, respectively. 
Then the right-hand side turns into that of the claimed 
identity \eqref{;id1d} and so does the left-hand side because 
$Z_{t}\bigl( \tr _{-\zeta }(B)\bigr) =Z_{t}$ by virtue of 
\pref{;pttrans}\thetag{iv}. The proof of \tref{;tmain1d} is 
completed.
\end{proof}

\begin{rem}\label{;rnotable2}
As for identity \eqref{;id1}, notice that, by the definition of $Z_{t}$, 
\begin{align*}
 \log (A_{t}/\T )=B_{t}-\log (\T /Z_{t}).
\end{align*} 
Identity \eqref{;ext0} reveals that 
\begin{align*}
 e^{B_{t}}\eqd \tau _{\cosh (x+B_{t})}(\hB ^{(\cosh x/Z_{t})})/Z_{t}, 
\end{align*}
and hence $B_{t}\eqd \log (\T /Z_{t})$ by recalling 
$\T =\tau _{\cosh (x+B_{t})}(\hB ^{(\cosh x/Z_{t})})$. 
Therefore the same remark as \rref{;rnotable} applies to 
the identity in law between the first coordinates in identity 
\eqref{;id1} as well: the difference of the two random variables 
$B_{t}$ and $\log (\T /Z_{t})$ with same law is substituted into 
$z$ of $\tr _{z}(B)$, which results in the process identical in 
law with $\tr _{0}(B)$. 
\end{rem}

\subsection{Proof of \lref{;lext}}\label{;ssplext}

In this subsection, we give a proof of \lref{;lext}. We denote by 
$\{ \Z _{s}\} _{s\ge 0}$ the natural filtration of the 
process $\{ Z_{s}\} _{s\ge 0}$. The proof hinges upon the 
following proposition due to Matsumoto and Yor, which is a 
consequence of the diffusion property of $\{ Z_{s}\} _{s\ge 0}$ 
investigated in detail in a pair of their papers \cite{myPI,myPII}; 
as they are dealing with Brownian motion with drift, we present 
here a particular case without drift.

\begin{prop}[\cite{myPI}, Proposition~1.7]\label{;pmy}
Conditionally on $\Z _{t}$ with $1/Z_{t}=u>0$, 
$B_{t}$ is distributed as a random variable $z_{u}$ 
whose law is given by 
\begin{align*}
 \pr (z_{u}\in dx)=\frac{1}{2K_{0}(u)}e^{-u\cosh x}\,dx,\quad x\in \R .
\end{align*}
Here $K_{0}$ is the modified Bessel function of the third kind 
(Macdonald function) of order $0$.
\end{prop}

The proof of \lref{;lext} is done by combining the above 
proposition with Proposition~2.3\thetag{ii} of \cite{har} asserting that, 
for every $x\in \R $ and $u>0$, 
\begin{equation}\label{;eqpkey3}
   \begin{split}
   &\left( 
   e^{z_{u}}\sinh x+\beta (e^{z_{u}}/u),\,e^{z_{u}}
   \right) \\
   &\eqd 
   \left( 
   \sinh (x+z_{u}),\,
   u\tau _{\cosh (x+z_{u})}(\hB ^{(u\cosh x)})
   \right) , 
   \end{split}
\end{equation}
where $\beta $ and $\hB $ are Brownian motions independent 
of $z_{u}$.

\begin{proof}[Proof of \lref{;lext}]
Let $f:\R \times (0,\infty )\to \R $ be a bounded 
measurable function and 
$F:C([0,t];\R )\to \R $ a bounded measurable functional.
By conditioning on $\Z _{t}$, 
\begin{align*}
 &\ex \!\left[ 
 f(e^{B_{t}}\sinh x+\beta (A_{t}),e^{B_{t}})F(Z_{s},s\le t)
 \right] \\
 &=\ex \!\left[ 
 \ex \!\left[ 
 f(e^{B_{t}}\sinh x+\beta (A_{t}),e^{B_{t}})\rmid| \Z _{t}
 \right] 
 F(Z_{s},s\le t)
 \right] .
\end{align*}
Notice that, by \pref{;pmy} and the relation 
$A_{t}=e^{B_{t}}Z_{t}$, the above conditional expectation 
given $\Z _{t}$ is equal a.s.\ to 
\begin{align*}
 \ex \!\left[ 
 f(e^{z_{u}}\sinh x+\beta (e^{z_{u}}/u),e^{z_{u}})
 \right] \big| _{u=1/Z_{t}}, 
\end{align*}
supposing $\beta $ is independent of $z_{u}$. 
Owing to \eqref{;eqpkey3}, the above expression is further rewritten as 
\begin{align*}
 \ex \!\left[ 
 f\bigl( \sinh (x+z_{u}),u\tau _{\cosh (x+z_{u})}(\hB ^{(u\cosh x)})\bigr) 
 \right] \bigg| _{u=1/Z_{t}}.
\end{align*}
Therefore, using \pref{;pmy} again, we have 
\begin{align*}
 &\ex \!\left[ 
 f(e^{B_{t}}\sinh x+\beta (A_{t}),e^{B_{t}})
 F(Z_{s},s\le t)
 \right] \\
 &=\ex \!\left[ 
 f\bigl( \sinh (x+B_{t}),
 \tau _{\cosh (x+B_{t})}(\hB ^{(\cosh x/Z_{t})})/Z_{t}\bigr) 
 F(Z_{s},s\le t)\right] ,
\end{align*}
$\hB $ being assumed to be independent of $B$. 
As $f$ and $F$ are arbitrary, the last equality entails \eqref{;ext}. 
\end{proof}

\section{Proofs of \tref{;tmain2} and \cref{;cmain2}}\label{;sptmain2}

In this section, we prove \tref{;tmain2} and its \cref{;cmain2}. 
Before proceeding to the proof of \tref{;tmain2}, we explain 
how \tref{;tmain1} and its extension \tref{;tmain1d} are connected to 
\tref{;tmain2}. We keep the notation used in \tref{;tmain1d}. 

Recall the fact that, for $a>0$ and $\mu \ge 0$, 
\begin{align}\label{;fact}
 \pr \!\left( 
 \tau _{a}(\db{\mu })\in du
 \right) 
 =\frac{a}{\sqrt{2\pi u^{3}}}\exp \left\{ 
 -\frac{(a-\mu u)^{2}}{2u}
 \right\} du,\quad u>0
\end{align}
(see, e.g., \cite[p.\,301, formula~2.2.0.2]{bs}). 
Suppose that $F:C([0,t];\R )\to [0,\infty )$ 
is bounded and continuous for the time being. 
Thanks to the independence of $B$ and $\hB $, the 
finite measure 
\begin{align*}
 \ex \!\left[ 
 F\bigl( \tr _{\log (A_{t}/\T )}(B)(s),s\le t\bigr) ;\,\log (A_{t}/\T )\in dz
 \right] 
\end{align*}
on $\R $, admits a density $f_{1}$  
with respect to the Lebesgue 
measure $dz$ expressed as 
\begin{align*}
 &f_{1}(z)\\
 &=\ex \!\left[ 
 \frac{\cosh (x+B_{t})}{\sqrt{2\pi e^{-z}A_{t}}}
 \exp \left\{ 
 -\frac{\bigl( 
 \cosh (x+B_{t})-e^{B_{t}-z}\cosh x
 \bigr) ^{2}}{2e^{-z}A_{t}}
 \right\} F\bigl( \tr _{z}(B)(s),s\le t\bigr) 
 \right] .
\end{align*}
Indeed, for every $z\in \R $, we have, by \eqref{;fact} and 
the independence of $B$ and $\hB $, 
\begin{align*}
 &\ex \!\left[ 
 F\bigl( \tr _{\log (A_{t}/\T )}(B)(s),s\le t\bigr) ;\,\log (A_{t}/\T )\le z
 \right] \\
 &=\ex \Biggl[ 
 \int _{e^{-z}A_{t}}^{\infty }du\,\frac{\cosh (x+B_{t})}{\sqrt{2\pi u^{3}}}
 \exp \left\{ 
 -\frac{\bigl( 
 \cosh (x+B_{t})-u\cosh x/Z_{t}
 \bigr) ^{2}}{2u}
 \right\} \\
 &\qquad \times F\bigl( \tr _{\log (A_{t}/u )}(B)(s),s\le t\bigr) 
 \Biggr] ,
\end{align*}
which is equal, by changing the variables inside the expectation 
with $u=e^{-y}A_{t},\,y\in \R $, and by 
Fubini's theorem, to 
\begin{align*}
 \int _{-\infty }^{z}f_{1}(y)\,dy. 
\end{align*}
This verifies the above expression of $f_{1}$. 
Notice that $f_{1}$ is continuous owing to the boundedness and 
continuity of $F$. In the same manner, we have 
\begin{align*}
 \ex \left[ 
 F(B_{s},s\le t);\,\zeta \in dz
 \right] =f_{2}(z)\,dz,\quad z\in \R ,
\end{align*}
with 
\begin{align*}
 f_{2}(z)=\ex \!\left[ 
 \frac{\cosh (x+z+B_{t})}{\sqrt{2\pi A_{t}}}\exp \left\{ 
 -\frac{\bigl( 
 \sinh (x+z+B_{t})-e^{B_{t}}\sinh x
 \bigr) ^{2}}{2A_{t}}
 \right\} F(B_{s},s\le t)
 \right] .
\end{align*}
See also the proof of \lref{;lpinned} as to the reasoning of 
the above derivation. It is clear that $f_{2}$ is continuous as well by 
the boundedness of $F$. By identity \eqref{;id1} in \tref{;tmain1d}, 
we have 
\begin{align}\label{;idf}
 f_{1}(z)=f_{2}(z)
\end{align}
for a.e.\ $z$, and hence for all $z$ by the continuity of $f_{1}$ and 
$f_{2}$. By the monotone convergence theorem, \eqref{;idf} 
holds true if $F$ is not bounded. Notice that, by 
\thetag{i} and \thetag{ii} of \pref{;pttrans},
\begin{align*}
 &f_{1}(z)\\
 &=\ex \Biggl[ 
 \frac{\cosh (x+z+\tr _{z}(B)(t))}
 {\sqrt{2\pi A_{t}(\tr _{z}(B))}}
 \exp \left\{ 
 -\frac{\bigl( 
 \cosh (x+z+\tr _{z}(B)(t))-e^{\tr _{z}(B)(t)}\cosh x
 \bigr) ^{2}}{2A_{t}(\tr _{z}(B))}
 \right\} \\
 &\qquad \times F\bigl( \tr _{z}(B)(s),s\le t\bigr) 
 \Biggr] .
\end{align*}
Therefore, if we replace $F$ by a functional of the form 
\begin{align*}
 \frac{\sqrt{2\pi A_{t}(\phi )}}{\cosh (x+z+\phi _{t})}
 \exp \left\{ 
 \frac{\bigl( \cosh (x+z+\phi _{t})-e^{\phi _{t}}\cosh x\bigr) ^{2}}
 {2A_{t}(\phi )}
 \right\} F(\phi _{s},s\le t)
\end{align*}
for $\phi \in C([0,t];\R )$, then we have, from \eqref{;idf}, 
\begin{align*}
 &\ex \!\left[ 
 F\bigl( \tr _{z}(B)(s),s\le t\bigr) 
 \right] \\
 &=\ex \Biggl[ 
 \exp \left\{ 
 \frac{
 \bigl( \cosh (x+z+B_{t})-e^{B_{t}}\cosh x\bigr) ^{2}
 -\bigl( \sinh (x+z+B_{t})-e^{B_{t}}\sinh x\bigr) ^{2}
 }{2A_{t}}
 \right\} \\
 &\qquad \times F(B_{s},s\le t)
 \Biggr] .
\end{align*}
Since 
\begin{align*}
 &\bigl( \cosh (x+z+B_{t})-e^{B_{t}}\cosh x\bigr) ^{2}
 -\bigl( \sinh (x+z+B_{t})-e^{B_{t}}\sinh x\bigr) ^{2}\\
 &=1-2e^{B_{t}}\cosh (z+B_{t})+e^{2B_{t}},
\end{align*}
we arrive, by the definition of $Z_{t}$, at relation \eqref{;t21} 
when $F$ is nonnegative and continuous. Then successive use of 
density and monotone class arguments extends $F$ to the class 
of all nonnegative measurable functionals as claimed in \tref{;tmain2}. 
In the subsequent subsection, we give a direct proof 
of \eqref{;t21} via \pref{;pmy}.

\begin{rem}\label{;rid2}
Thanks to \pref{;pttrans}\thetag{iii}, identity \eqref{;id2} in the 
proof of \tref{;tmain1d} may also be proven by replacing $F$ by 
$F\circ \tr _{-z}$ in \eqref{;idf}.
\end{rem}

\subsection{Proof of \tref{;tmain2}}\label{;ssptmain2}

In what follows, we fix $z\in \R $. We begin with 

\begin{lem}\label{;lstart}
It holds that, for every bounded measurable function 
$f:\R \to \R $ and for every bounded measurable functional 
$F:C([0,t];\R )\to \R $, 
\begin{align*}
 \ex \!\left[ 
 f(B_{t}-z)F(Z_{s},s\le t)
 \right] 
 =\ex \!\left[ 
 \exp \left\{ 
 \frac{\cosh B_{t}-\cosh (z+B_{t})}{Z_{t}}
 \right\} f(B_{t})F(Z_{s},s\,\le t)
 \right] .
\end{align*}
\end{lem}

\begin{proof}
We start from the right-hand side, by conditioning on $\Z _{t}$, 
to rewrite it as 
\begin{align*}
 \ex \!\left[ 
 \ex \!\left[ \exp \left\{ 
 \frac{\cosh B_{t}-\cosh (z+B_{t})}{Z_{t}}
 \right\} f(B_{t})\rmid| \Z_{t}\right] F(Z_{s},s\,\le t)
 \right] .
\end{align*}
Then, by \pref{;pmy}, the conditional expectation in the 
integrand is equal a.s.\ to 
\begin{align*}
 \ex \!\left[ 
 \exp \left\{ 
 u\bigl( 
 \cosh z_{u}-\cosh (z+z_{u})
 \bigr) 
 \right\} f(z_{u})
 \right] \Big| _{u=1/Z_{t}},
\end{align*}
which is computed as 
\begin{align*}
 &\frac{1}{2K_{0}(u)}\int _{\R }dx\,e^{-u\cosh x}
 \exp \left\{ 
 u\bigl( 
 \cosh x-\cosh (z+x)
 \bigr) 
 \right\} f(x)\\
 &=\frac{1}{2K_{0}(u)}\int _{\R }dx\,e^{-u\cosh (z+x)}f(x)\\
 &=\ex \!\left[ f(z_{u}-z)\right] ,
\end{align*}
with $1/Z_{t}$ inserted into $u$. Therefore we see that the 
right-hand side of the claimed equality turns into 
\begin{align*}
 \ex \!\left[ \ex \!\left[ 
 f(B_{t}-z)\rmid| \Z_{t}\right] F(Z_{s},s\le t)
 \right] 
\end{align*}
and hence into the left-hand side.
\end{proof}

Now that we are convinced that 
\begin{align*}
 \pr _{t,z}:=\exp \left\{ 
 \frac{\cosh B_{t}-\cosh (z+B_{t})}{Z_{t}}
 \right\} \pr 
\end{align*}
determines a probability measure, we may translate \lref{;lstart} 
into the following identity in law: 
\begin{align}\label{;idlaw}
 \left( 
 B_{t}-z,\,\{ Z_{s}\} _{0\le s\le t}
 \right) \eqd _{\pr _{t,z}}\left( 
 B_{t},\,\{ Z_{s}\} _{0\le s\le t}
 \right) .
\end{align}
Here the notation $\eqd _{\pr _{t,z}}$ indicates that 
the law of the right-hand side of a claimed identity is 
considered under $\pr _{t,z}$. With this notation and from 
\eqref{;idlaw}, the proof of 
\tref{;tmain2} proceeds along the same lines as in the 
proof of \tref{;tmain1d}.

\begin{proof}[Proof of \tref{;tmain2}]
By \eqref{;idlaw} and \eqref{;deria}, we have 
\begin{align*}
 \left( 
 B_{t}-z,\,Z_{t},\,\left\{ 
 \frac{1}{A_{s}}-\frac{1}{A_{t}}
 \right\} _{0<s\le t}
 \right) 
 \eqd _{\pr _{t,z}}
 \left( 
 B_{t},\,Z_{t},\,\left\{ 
 \frac{1}{A_{s}}-\frac{1}{A_{t}}
 \right\} _{0<s\le t}
 \right) ,
\end{align*}
and hence, by the relation $A_{t}=e^{B_{t}}Z_{t}$, 
\begin{align*}
 \left( 
 e^{-z}A_{t},\,\left\{ 
 \frac{1}{A_{s}}-\frac{1}{A_{t}}
 \right\} _{0<s\le t}
 \right) 
 \eqd _{\pr _{t,z}}
 \left( 
 A_{t},\,\left\{ 
 \frac{1}{A_{s}}-\frac{1}{A_{t}}
 \right\} _{0<s\le t}
 \right) .
\end{align*}
Therefore we have 
\begin{align*}
 \left\{ 
 \frac{1}{A_{s}}-\frac{1}{A_{t}}+\frac{e^{z}}{A_{t}}
 \right\} _{0<s\le t}
 \eqd _{\pr _{t,z}}
 \left\{ 
 \frac{1}{A_{s}}
 \right\} _{0<s\le t}, 
\end{align*}
from which it follows that, thanks to \pref{;pttrans}\thetag{ii}, 
\begin{align*}
 \left\{ A_{s}(\tr _{z}(B))(s)\right\} _{0\le s\le t}
 \eqd _{\pr _{t,z}}
 \{ A_{s}\} _{0\le s\le t}. 
\end{align*}
This proves relation \eqref{;t21}. 
By virtue of \pref{;pttrans}\thetag{iii}, relation \eqref{;t22} follows 
by replacing $F$ by $F\circ \tr _{-z}$ in \eqref{;t21}. 
The equivalence between \eqref{;t21} and \eqref{;t22} is also clear.
\end{proof}

We give a comment on the consistency of relation \eqref{;t21} 
under the time reversal \eqref{;trev}. Recall that the law of 
$\{ B_{s}\} _{0\le s\le t}$ is invariant under $R$: 
\begin{align}\label{;itrev}
 \left\{ R(B)(s)\right\} _{0\le s\le t}
 \eqd 
 \{ B_{s}\} _{0\le s\le t}.
\end{align}

\begin{rem}\label{;rtrev} 
By \eqref{;t21}, and by $R(B)(t)=-B_{t}$ and 
$Z_{t}(R(B))=Z_{t}$ because of \lref{;ltrev}, we have, 
for every $z\in \R $, 
\begin{align*}
 &\ex \!\left[ 
 F\bigl( \tr _{-z}(B)(s),s\le t\bigr) 
 \right] \\
 &=\ex \!\left[ 
 \exp \left\{ 
 \frac{\cosh R(B)(t)-\cosh \bigl( -z+R(B)(t)\bigr) }
 {Z_{t}(R(B)) }
 \right\} F\bigl (R(B)(s),s\le t\bigr) 
 \right] \\
 &=\ex \!\left[ 
 \exp \left\{ 
 \frac{\cosh B_{t}-\cosh (z+B_{t})}{Z_{t}}
 \right\} F\bigl (R(B)(s),s\le t\bigr) 
 \right] ,
\end{align*}
which is equal to 
\begin{align*}
 \ex \!\left[ 
 F\bigl( R(\tr _{z}(B))(s),s\le t\bigr) 
 \right] 
\end{align*}
by \eqref{;t21} again. The above observation entails that 
\begin{align*}
 \tr _{-z}(B)\eqd R\bigl( \tr _{z}(B)\bigr) ,
\end{align*}
which is indeed the case since, by \pref{;pttrans}\thetag{v} and 
\eqref{;itrev}, 
\begin{align*}
 R\bigl( \tr _{z}(B)\bigr) &=\tr _{-z}\bigl( R(B)\bigr) \\
 &\eqd \tr _{-z}(B).
\end{align*}
\end{rem}

\subsection{Proof of \cref{;cmain2}}\label{;sspcmain2}

In this subsection, we prove \cref{;cmain2} as well as explore some 
related facts. We begin with the following lemma.

\begin{lem}\label{;lexpm}
For every $\mu \in \R $ and $t>0$, we have 
\begin{align*}
 \ex \left[ 
 \exp \left( 
 \frac{\theta }{2\da{\mu }_{t}}
 \right) 
 \right] <\infty \quad 
 \text{for all }\theta <1.
\end{align*}
\end{lem}

In order to prove the lemma, we recall from \cite[equation~(1.5)]{ag} 
that the law of $A_{t}$ admits the density function expressed as 
\begin{align}\label{;Ad}
\exp \left( 
 \frac{\pi ^{2}}{8t}
 \right) 
 \ex \!\left[ 
 \frac{\cosh B_{t}}{\sqrt{2\pi v^{3}}}
 \exp \left( 
 -\frac{\cosh ^{2}B_{t}}{2v}
 \right) \cos \left( \frac{\pi }{2t}B_{t}\right) 
 \right] ,\quad v>0.
\end{align}
We refer the reader to \cite{hari} for an account of 
explicit expressions of the laws of the functionals 
$\da{\mu }_{t}$; see also \cite{my2003} and references therein.

\begin{proof}[Proof of \lref{;lexpm}]
From \eqref{;Ad}, it follows readily that 
\begin{align}\label{;Adbd}
 \frac{\pr (1/A_{t}\in dv)}{dv}
 \le \frac{C}{\sqrt{v}}\exp \left( -\frac{v}{2}\right) \quad 
 \text{for all $v>0$,}
\end{align}
for some positive constant $C$ depending only on $t$.
Pick $\theta <1$ arbitrarily and let $p>1$ be such that 
$p\theta <1$. Then, by the Cameron--Martin relation and 
H\"older's inequality, 
\begin{align*}
 \ex \left[ 
 \exp \left( 
 \frac{\theta }{2\da{\mu }_{t}}
 \right) 
 \right] 
 &=\ex \left[ 
 \exp \left( 
 \frac{\theta }{2A_{t}}
 \right) e^{\mu B_{t}}
 \right] e^{-\mu ^{2}t/2}\\
 &\le \ex \left[ 
 \exp \left( 
 \frac{p\theta }{2A_{t}}
 \right) 
 \right] ^{1/p}\exp \left\{ 
 \frac{\mu ^{2}t}{2(p-1)}
 \right\} ,
\end{align*}
which is finite in view of \eqref{;Adbd}.
\end{proof}

Using the above lemma, we prove \cref{;cmain2}.

\begin{proof}[Proof of \cref{;cmain2}]
Fix $\mu >0$. Since there is nothing to prove when $z=0$, 
we let $z\neq 0$. It suffices to prove the claim for functionals 
$F$ of the form 
\begin{align}\label{;cyl}
 F(\phi _{s},s\ge 0)=f(\phi _{t_{1}},\ldots ,\phi _{t_{n}}),\quad 
 \phi \in C([0,\infty );\R ),
\end{align}
with a positive integer $n$, $0\le t_{1}\le \cdots \le t_{n}$, and 
a bounded continuous function $f:\R ^{n}\to \R $.
We pick such an $F$ and let $t\ge t_{n}$. Then, by replacing 
$F$ in \eqref{;t21} by 
$F(\phi _{s},s\ge 0)e^{-\mu \phi _{t}-\mu ^{2}t/2}$, the 
Cameron--Martin relation entails that 
\begin{align*}
 &\ex \!\left[ 
 F\bigl( 
 \tr _{z}(\db{-\mu })(s),s\ge 0
 \bigr) 
 \right] e^{\mu z}\\
 &=\ex \!\left[ 
 \exp \left\{ 
 \frac{\cosh \db{-\mu }_{t}-\cosh \bigl( z+\db{-\mu }_{t}\bigr) }
 {\dz{-\mu }_{t}}
 \right\} F\bigl( \db{-\mu }_{s},s\ge 0\bigr) 
 \right] .
\end{align*}
Here the factor $e^{\mu z}$ on the left-hand side 
comes from property~\thetag{i} in \pref{;pttrans}. 
Because of the continuity and boundedness of $f$ determining 
the functional $F$, the bounded convergence theorem entails that, 
as $t\to \infty $, the left-hand side converges to  
\begin{align*}
  \ex \!\left[ 
 F\bigl( \tr ^{*}_{z}(\db{-\mu })(s),s\ge 0\bigr) 
 \right] e^{\mu z}.
\end{align*}
(Recall \eqref{;ttrans} and \eqref{;ttransd} of the definitions of 
$\tr _{z}$ and $\tr ^{*}_{z}$, respectively; in particular, the 
definition of the former involves parameter $t$.) 
We turn to the right-hand side. 
By the definition of $\dz{-\mu }_{t}$, the integrand is rewritten as 
\begin{align}\label{;integrandr}
 \exp \left\{ 
 \frac{e^{2\db{-\mu }_{t}}(1-e^{z})+1-e^{-z}}{2\da{-\mu }_{t}}
 \right\} F\bigl( \db{-\mu }_{s},s\ge 0\bigr) , 
\end{align}
which, as $t\to \infty $, converges a.s.\ to 
\begin{align}\label{;integrandl}
 \exp \left( \frac{1-e^{-z}}{2\da{-\mu }_{\infty }}\right) 
 F\bigl( \db{-\mu }_{s},s\ge 0\bigr) .
\end{align}
Therefore, if we know that the family of random variables 
\eqref{;integrandr}, indexed by $t\ge t_{n}$, is uniformly integrable, 
then the right-hand side  of the last equality converges to the 
expectation of \eqref{;integrandl} and hence relation \eqref{;ec1} 
follows. To this end, we denote by $X_{t}$ the random variable 
given in \eqref{;integrandr}. To see the uniform integrability of 
$\{ X_{t}\} _{t\ge t_{n}}$, we divide the case into two cases. 
In the case $z>0$, 
\begin{align*}
 |X_{t}|\le M\exp \left( \frac{1-e^{-z}}{2\da{-\mu }_{t_{n}}}\right) 
 \quad \text{for any $t\ge t_{n}$,}
\end{align*}
where $M$ is the supremum of $|f|$. Therefore, by \lref{;lexpm}, 
there exists $p>1$ such that 
\begin{align*}
 \sup _{t\ge t_{n}}\ex \!\left[ |X_{t}|^{p}\right] <\infty ,
\end{align*}
from which the desired uniform integrability follows. In the case 
$z<0$, the same estimate holds as well; indeed, 
\begin{align*}
 |X_{t}|&\le M\exp \left\{ 
 \frac{e^{2\db{-\mu }_{t}}(1-e^{z})}{2\da{-\mu }_{t}}
 \right\} \\
 &=M\exp \left\{ 
 \frac{1-e^{z}}{2A_{t}\bigl( R(\db{-\mu })\bigr) }
 \right\} ,
\end{align*}
with $R$ the time reversal on $C([0,t];\R )$ in 
\eqref{;trev}, and hence picking $p>1$ such that 
$p(1-e^{z})<1$, we have, for any $t\ge t_{n}$, 
\begin{align*}
 \ex \!\left[ 
 |X_{t}|^{p}
 \right] 
 &\le M^{p}\ex \!\left[ 
 \exp \left\{ 
 \frac{p(1-e^{z})}{2\da{\mu }_{t}}
 \right\} 
 \right] \\
 &\le M^{p}\ex \!\left[ 
 \exp \left\{ 
 \frac{p(1-e^{z})}{2\da{\mu }_{t_{n}}}
 \right\} 
 \right] \\
 &<\infty 
\end{align*}
by \lref{;lexpm}. Here, for the first line, we used the fact that 
\begin{align*}
 \left\{ R(\db{-\mu })(s)\right\} _{0\le s\le t}
 \eqd \left\{ \db{\mu }_{s}\right\} _{0\le s\le t}, 
\end{align*}
which follows from \eqref{;itrev}. The proof of relation 
\eqref{;ec1} completes. Relation \eqref{;ec2} is immediate 
from \eqref{;ec1} and the fact that, 
for any $z\in \R $ and for any $\phi \in C([0,\infty );\R )$ with 
$A_{\infty }(\phi )<\infty $, 
\begin{align*}
 \tr ^{*}_{-z}\bigl( \tr ^{*}_{z}(\phi )\bigr) =\phi ,
\end{align*}
proof of which is done in the same way as that of 
\pref{;pttrans}\thetag{iii}, hence omitted. 
The equivalence of the two relations \eqref{;ec1} and 
\eqref{;ec2} is also clear from the above fact.
\end{proof}

\begin{rem}\label{;rscheffe}
We have not used Dufresne's identity \eqref{;duf} 
in the above proof as pointed out in \rref{;rduf} that 
it can be deduced from \cref{;cmain2}. If we have Dufresne's 
identity at our disposal, the interchangeability of the order of limit 
and expectation seen above may be verified by Scheff\'e's 
lemma (see, e.g., \cite[p.~55]{will}) since, for any $t>0$, 
\begin{align*}
 \ex \!\left[ 
 \exp \left\{ 
 \frac{\cosh \db{-\mu }_{t}-\cosh \bigl( z+\db{-\mu }_{t}\bigr) }
 {\dz{-\mu }_{t}}
 \right\} \right] 
 =\ex \!\left[ 
 \exp \left( \frac{1-e^{-z}}{2\da{-\mu }_{\infty }}\right) 
 \right] 
\end{align*}
(in fact, both sides are equal to $e^{\mu z}$), which, together with 
their a.s.\ convergence, ensures the $L^{1}$-convergence of 
the integrands.
\end{rem}

Recall from \cite{dmy} a family $\{ T_{\al }\} _{\al >0}$ of path 
transformations defined by 
\begin{align}\label{;tra}
 T_{\al }(\phi )(s):=\phi _{s}-\log \left\{ 
 1+\al A_{s}(\phi )
 \right\} ,\quad s\ge 0,
\end{align}
for $\phi \in C([0,\infty );\R )$. The following identity 
concerning the laws of $\db{\pm \mu }$ is shown in 
Matsumoto--Yor \cite[Theorem~2.1]{myPIII}: when $\mu >0$, 
\begin{align}\label{;opp}
 \left\{ \db{-\mu }_{s}\right\} _{s\ge 0}
 \eqd \left\{ 
 T_{2\ga _{\mu }}(\db{\mu })(s)
 \right\} _{s\ge 0}.
\end{align}
Here $\ga _{\mu }$ is a gamma random variable with parameter 
$\mu $ independent of $B$. We will show that \cref{;cmain2} 
may also be obtained by \eqref{;opp}. First we observe the 
following relationship between the two transformations 
\eqref{;ttransd} and \eqref{;tra}: 

\begin{lem}\label{;rel}
It holds that, for every $z\in \R $ and $\al >0$, 
\begin{align*}
 \tr ^{*}_{z}\bigl( T_{\al }(\phi )\bigr) (s)
 =T_{\al e^{z}}(\phi )(s),\quad s\ge 0,
\end{align*}
for any $\phi \in C([0,\infty );\R )$ such that 
$A_{\infty }(\phi )=\infty $.
\end{lem}

\begin{proof}
By the definition of $T_{\al }$, we have, for every 
$\phi \in C([0,\infty );\R )$, 
\begin{align*}
 A_{s}(T_{\al }(\phi ))=\frac{A_{s}(\phi )}{1+\al A_{s}(\phi )},\quad 
 s\ge 0, 
\end{align*}
and hence when $A_{\infty }(\phi )=\infty $, 
\begin{align}
 A_{\infty }(T_{\al }(\phi ))
 \equiv \lim _{s\to \infty }A_{s}(T_{\al }(\phi ))
 &=\frac{1}{\al }, \label{;ainfty}
\end{align}
which is finite. Therefore, by the definition of $\tr ^{*}_{z}$, 
we have for every $s\ge 0$, 
\begin{align*}
 \tr ^{*}_{z}\bigl( T_{\al }(\phi )\bigr) (s)
 &=\phi _{s}-\log \left\{ 1+\al A_{s}(\phi )\right\} 
 -\log \left\{ 
 1+\frac{A_{s}(\phi )}{1+\al A_{s}(\phi )}\al (e^{z}-1)
 \right\} \\
 &=\phi _{s}-\log \left\{ 
 1+\al A_{s}(\phi )+A_{s}(\phi )\al (e^{z}-1)
 \right\} \\
 &=\phi _{s}-\log \left\{ 1+\al e^{z}A_{s}(\phi )\right\} ,
\end{align*}
which is the claim.
\end{proof}

\begin{proof}[Proof of \cref{;cmain2} via \eqref{;opp}]
Here we prove the latter relation \eqref{;ec2}. 
By identity \eqref{;opp} and \lref{;rel}, the right-hand side 
of \eqref{;ec2} is written as 
\begin{align*}
 &e^{-\mu z}\ex \!\left[ 
 \exp \left\{ 
 \frac{1-e^{-z}}{2A_{\infty }\bigl( T_{2\ga _{\mu }}(\db{\mu })\bigr) }
 \right\} 
 F\bigl( T_{2e^{-z}\ga _{\mu }}(\db{\mu })(s),s\ge 0\bigr) 
 \right] \\
 &=e^{-\mu z}\ex \!\left[ 
 \exp \left\{ (1-e^{-z})\ga _{\mu }\right\} 
 F\bigl( T_{2e^{-z}\ga _{\mu }}(\db{\mu })(s),s\ge 0\bigr) 
 \right] ,
\end{align*}
where the equality is due to \eqref{;ainfty}. By the independence of 
$B$ and $\ga _{\mu }$, and by Fubini's theorem, the above 
expression is computed as 
\begin{align*}
 &\frac{e^{-\mu z}}{\Gamma (\mu )}\int _{0}^{\infty }du\,
 u^{\mu -1}e^{-u}\exp \left\{ (1-e^{-z})u\right\} 
 \ex \!\left[ 
 F\bigl( T_{2e^{-z}u}(\db{\mu })(s),s\ge 0\bigr) \right] \\
 &=\frac{1}{\Gamma (\mu )}\int _{0}^{\infty }dv\,v^{\mu -1}e^{-v}
 \ex \!\left[ 
 F\bigl( T_{2v}(\db{\mu })(s),s\ge 0\bigr) 
 \right] \\
 &=\ex \!\left[ 
 F\bigl( T_{2\ga _{\mu }}(\db{\mu })(s),s\ge 0\bigr) 
 \right] ,
\end{align*}
where we have changed the variables with $u=e^{z}v$ for 
the second line. Owing to \eqref{;opp}, the last expectation 
coincides with the left-hand side of \eqref{;ec2}.
\end{proof}

\begin{rem}\label{;rec1}
On the other hand, applying the same reasoning as above to 
the right-hand side of \eqref{;ec1} leads to the identity 
\begin{align*}
 \left\{ \tr ^{*}_{z}(\db{-\mu })(s)\right\} _{s\ge 0}
 \eqd \left\{ T_{2e^{z}\ga _{\mu }}(\db{\mu })(s)\right\} _{s\ge 0},
\end{align*}
with $\ga _{\mu }$ independent of $B$. The case $z=0$ is 
identity \eqref{;opp}; letting $z\to -\infty $ leads to 
\begin{align*}
 \left\{ 
 \db{-\mu }_{s}-\log \left( 
 1-\frac{\da{-\mu }_{s}}{\da{-\mu }_{\infty }}
 \right) 
 \right\} _{s\ge 0}
 \eqd \left\{ \db{\mu }_{s}\right\} _{s\ge 0},
\end{align*}
which is found in \cite[Proposition~3.1]{myPIII}.
\end{rem}

\cref{;cmain2} enables us to obtain identity \eqref{;inv}.
\begin{prop}\label{;pintr}
Let $\mu >0$ and suppose $B$ and $\ga _{\mu }$ are independent. 
Then we have identity \eqref{;inv}; that is, the process  
\begin{align*}
 \db{-\mu }_{s}-\log \left\{ 
 1+\da{-\mu }_{s}\left( 
 2\ga _{\mu }-\frac{1}{\da{-\mu }_{\infty }}
 \right) 
 \right\} ,\quad s\ge 0,
\end{align*}
is identical in law with $\db{-\mu }$.
\end{prop}

\begin{proof}
In what follows, we write $X$ for 
$1/\bigl( 2\da{-\mu }_{\infty }\bigr) $ for simplicity. 
We suppose that $F:C([0,\infty );\R )\to \R $ is 
bounded and continuous, and we take a bounded 
measurable function $f:(0,\infty )\to \R $ arbitrarily. 
By noting (cf.\ \pref{;pttrans}\thetag{ii}) that 
\begin{align*}
 \frac{1}{A_{s}\bigl( \tr ^{*}_{z}(\db{-\mu })\bigr) }
 &=\frac{1}{\da{-\mu }_{s}}
 +\frac{e^{z}-1}{\da{-\mu }_{\infty }}\\
 &\xrightarrow[s\to \infty ]{}
 \frac{e^{z}}{\da{-\mu }_{\infty }},
\end{align*}
it follows from \eqref{;ec1} that 
\begin{align*}
 &\ex \!\left[ 
 F\bigl( \tr ^{*}_{z}(\db{-\mu })(s),s\ge 0\bigr) f(X)
 \right] \\
 &=e^{-\mu z}\ex \!\left[ 
 \exp \left\{ (1-e^{-z})X\right\}  
 F\bigl( \db{-\mu }_{s},s\ge 0\bigr) f(e^{-z}X)
 \right] .
\end{align*}
By the fact that $X\eqd \ga _{\mu }$, the left-hand side 
is disintegrated as 
\begin{align*}
 \frac{1}{\Gamma (\mu )}\int _{0}^{\infty }du\,
 u^{\mu -1}e^{-u}f(u)\ex \!\left[ 
 F\bigl( \tr ^{*}_{z}(\db{-\mu })(s),s\ge 0\bigr) \rmid| X=u
 \right] .
\end{align*}
On the other hand, 
by conditioning on $X=u$, and by using the same change of 
the variables as in the last proof, the right-hand side 
is expressed as 
\begin{align*}
 \frac{1}{\Gamma (\mu )}\int _{0}^{\infty }dv\,v^{\mu -1}e^{-v}
 f(v)\ex \!\left[ 
 F\bigl( \db{-\mu }_{s},s\ge 0\bigr) \rmid| X=e^{z}v 
 \right] .
\end{align*}
Since the last two expressions agree for any $f$, we conclude that 
\begin{align*}
 \ex \!\left[ 
 F\bigl( \tr ^{*}_{z}(\db{-\mu })(s),s\ge 0\bigr) \rmid| X=u
 \right] 
 =\ex \!\left[ 
 F\bigl( \db{-\mu }_{s},s\ge 0\bigr) \rmid| X=e^{z}u 
 \right] 
\end{align*}
for a.e.\ $u>0$. By the continuity of $F$, we may assume that 
each side admits a continuous version (see \rref{;rcont1}) and hence 
that the above relation holds for all $u>0$. Given $v>0$, 
we insert $\log (v/u)$ into $z$ to obtain 
\begin{align*}
 \ex \!\left[ 
 F\bigl( \tr ^{*}_{\log (v/u)}(\db{-\mu })(s),s\ge 0\bigr) \rmid| X=u
 \right] 
 =\ex \!\left[ 
 F\bigl( \db{-\mu }_{s},s\ge 0\bigr) \rmid| X=v 
 \right] ,
\end{align*}
which is valid for all $u,v>0$. Integrating both sides with respect to 
the measure $\pr (X\in du)\pr (X\in dv)$ over 
$(0,\infty )^{2}$, we reach the conclusion owing to the 
arbitrariness of $F$.
\end{proof}

\begin{rem}
Let $h$ be an arbitrary positive measurable function on 
$(0,\infty )^{2}$ and replace in the last displayed equation 
$F$ by a functional of the form 
\begin{align*}
 F\bigl( \tr ^{*}_{\log (uh(u,v))}(\phi )(s),s\ge 0\bigr) 
\end{align*} 
for $\phi \in C([0,\infty );\R )$ with $A_{\infty }(\phi )<\infty $. 
Then, using a semigroup property as in \pref{;pttrans}\thetag{iii} 
of $\tr^{*}_{z},\,z\in \R $, we may have the following generalization 
of \pref{;pintr}: the two processes 
\begin{align*}
 &\db{-\mu }_{s}-\log \left\{ 
 1+\frac{\da{-\mu }_{s}}{\da{-\mu }_{\infty }}
 \left( \ga _{\mu }h\biggl( 
 \frac{1}{2\da{-\mu }_{\infty }},\ga _{\mu }
 \biggr) -1\right) 
 \right\} ,\quad s\ge 0,\\
 &\db{-\mu }_{s}-\log \left\{ 
 1+\frac{\da{-\mu }_{s}}{\da{-\mu }_{\infty }}
 \left( \ga _{\mu }h\biggl( 
 \ga _{\mu },\frac{1}{2\da{-\mu }_{\infty }}
 \biggr) -1\right) 
 \right\} ,\quad s\ge 0,
\end{align*}
are identical in law; taking either $h(u,v)=1/u$ or $h(u,v)=1/v$ 
yields \pref{;pintr}. The above identity in law may also be 
verified by means of identity \eqref{;opp}. Indeed, if we let 
$\hat{\ga }_{\mu }$ be a copy of $\ga _{\mu }$ such that 
$B$, $\ga _{\mu }$ and $\hat{\ga }_{\mu }$ 
are independent, and if we replace $\db{-\mu }$ by 
$T_{2\hat{\ga }_{\mu }}(\db{\mu })$ in the above two processes, 
then we obtain the process 
$T_{2\ga _{\mu }\hat{\ga }_{\mu }
h(\hat{\ga }_{\mu },\ga _{\mu })}(\db{\mu })$ 
for the former and likewise 
$T_{2\ga _{\mu }\hat{\ga }_{\mu }
h(\ga _{\mu },\hat{\ga }_{\mu })}(\db{\mu })$ 
for the latter, and these two are clearly identical in law 
because of the independence of the three elements involved.
\end{rem}

\begin{rem}\label{;rcont1}
To see that there exists a continuous version for a function 
of the form 
\begin{align}\label{;cexp}
 \ex \!\left[ 
 F\bigl( \db{-\mu }_{s},s\ge 0\bigr) \rmid| \da{-\mu }_{\infty }=u 
 \right] ,\quad u>0,
\end{align}
with $F:C([0,\infty );\R )\to \R $ a bounded continuous functional, 
we recall Lamperti's relation for 
the exponential functionals of $\db{-\mu }$: 
there exists a $2(1-\mu )$-dimensional Bessel process $\rho $ 
starting from $1$, such that 
\begin{align}\label{;lamp}
 e^{\db{-\mu }_{s}}=\rho _{\da{-\mu }_{s}},\quad s\ge 0.
\end{align}
See, e.g., \cite[Section~3]{myPIII}. In particular, letting 
$s\to \infty $ on both sides yields the relation 
$\da{-\mu }_{\infty }=\tau _{0}(\rho )$. By setting 
\begin{align}\label{;inval}
 \al _{s}(\rho )
 =\inf \left\{ 
 t\ge 0;\,\int _{0}^{t}
 \frac{dv}{\rho _{v}^{2}}\ind _{\{ v<\tau _{0}(\rho )\} }>s
 \right\} ,\quad s\ge 0,
\end{align}
Lamperti's relation \eqref{;lamp} entails the 
representation of $\db{-\mu }$ in terms of $\rho $:
\begin{align*}
 \db{-\mu }_{s}=\log \rho _{\al _{s}(\rho )},\quad s\ge 0.
\end{align*}
Note that, conditionally on $\tau _{0}(\rho )=u>0$, the process 
$\rho $ is a $2(\mu +1)$-dimensional Bessel bridge of duration 
$u$ ending at $0$ 
(see, e.g., \cite[Exercises~XI.1.23 and XI.3.12]{ry}), and hence 
identical in law with 
\begin{align*}
 \left( 1-\frac{s}{u}\right) \hat{\rho }_{us/(u-s)},\quad 
 0\le s<u, 
\end{align*}
where $\hat{\rho }$ is a $2(\mu +1)$-dimensional Bessel process 
starting from $1$ (see, e.g., \cite[Exercise~XI.3.6]{ry}). These 
facts amount to the identity between the law 
of $\db{-\mu }$ conditioned on $\da{-\mu }_{\infty }=u$ and 
that of the process  
\begin{align}\label{;lampd}
 \log \left\{ 
 \frac{u}{u+\al _{s}(\hat{\rho })}\hat{\rho }_{\al _{s}(\hat{\rho })}
 \right\} ,\quad s\ge 0,
\end{align}
with $\al _{s}(\hat{\rho })$ defined through \eqref{;inval}, 
$\hat{\rho }$ replacing $\rho $ (in this case, 
$\tau _{0}(\hat{\rho })=\infty $ a.s.). It is clear that 
the law of \eqref{;lampd} depends on $u$ continuously. 
Therefore the function in \eqref{;cexp} admits a continuous 
version. 
\end{rem}

\begin{rem}
Lamperti's relation for the exponential functionals of $\db{\mu }$ 
reveals that \eqref{;lampd} is identical in law with the process 
\begin{align*}
 \db{\mu }_{s}-\log \left( 1+\da{\mu }_{s}/u\right) ,\quad s\ge 0,
\end{align*}
which, together with $1/\da{-\mu }_{\infty }\eqd 2\ga _{\mu }$, 
entails \eqref{;opp}. A proof of identity \eqref{;opp} along this 
line may also be found in \cite[Section~3]{myPIII}.
\end{rem}

\section{Proof of \tref{;tmain3}}\label{;sptmain3}

This section is devoted to the proof of \tref{;tmain3}. 
For each $z\in \R $, we denote by $b^{z}=\{ b^{z}_{s}\} _{0\le s\le t}$ 
a Brownian bridge of duration $t$, starting from $0$ and 
ending at $z$. Let $F:C([0,t];\R )\to [0,\infty )$ be a bounded 
continuous functional and $f:(0,\infty ) \to [0,\infty )$ a bounded 
continuous function.

\begin{lem}\label{;lpinned}
Fix $z\in \R $ and let $\hB =\{ \hB _{s}\} _{s\ge 0}$ be a 
Brownian motion independent of $b^{z}$. Then it holds that, 
for any $x\in \R $, 
\begin{align*}
 &\ex \!\left[ 
 F\bigl( 
 \tr _{B_{t}-z}(B)(s),s\le t
 \bigr) f(A_{t})\frac{\cosh (x+z)}{\sqrt{2\pi A_{t}}}
 \exp \left\{ 
 -\frac{\bigl( 
 \sinh (x+z)-e^{B_{t}}\sinh x
 \bigr) ^{2}}{2A_{t}}
 \right\} 
 \right] \\
 &=\ex \!\left[ 
 F(b^{z}_{s},s\le t)f\bigl( 
 \tau _{\cosh (x+z)}(\hB ^{(\cosh x/Z_{t}(b^{z}))})
 \bigr) 
 \right] \frac{1}{\sqrt{2\pi t}}\exp \left( 
 -\frac{z^{2}}{2t}
 \right) .
\end{align*}
\end{lem}

\begin{proof}
For each fixed $x\in \R $, we denote by the left-hand and 
right-hand sides of the claimed equality by $g_{1}(z)$ and 
$g_{2}(z)$, respectively. Notice that $g_{1}$ and $g_{2}$ are 
continuous in $z$ by the assumptions on $F$ and $f$ 
(see \rref{;rcont2}). We keep the notation of \tref{;tmain1d}.
By identity \eqref{;id1d}, we know that the two finite measures 
\begin{align}\label{;twofm}
 \ex \!\left[ 
 F\bigl( \tr _{B_{t}-z}(B)(s),s\le t\bigr) f(A_{t});\,B_{t}+\zeta \in dz
 \right] , && 
 \ex \!\left[ F(B_{s},s\le t)f(\T );B_{t}\in dz\right] ,
\end{align}
on $\R $, agree. It is clear that the latter admits the density function 
$g_{2}$ with respect to the Lebesgue measure. On the other hand, 
for an arbitrary $z\in \R $, 
noting that the event that $B_{t}+\zeta \le z$ is equal to the 
event that $e^{B_{t}}\sinh x+\beta (A_{t})\le \sinh (x+z)$ 
by the definition of $\zeta $, we see that, by the independence of 
$B$ and $\beta $, and by Fubini's theorem,  
\begin{align*}
 &\ex \!\left[ 
 F\bigl( \tr _{B_{t}-(B_{t}+\zeta )}(B)(s),s\le t\bigr) 
 f(A_{t});\,B_{t}+\zeta \le z
 \right] \\
 &=\int _{-\infty }^{\sinh (x+z)}dy\,\ex \!\Biggl[ 
 F\bigl( \tr _{B_{t}-\argsh y+x}(B)(s),s\le t\bigr) f(A_{t})\\
 &\qquad \qquad \qquad \qquad \times \frac{1}{\sqrt{2\pi A_{t}}}
 \exp \left\{ 
 -\frac{\bigl( y-e^{B_{t}}\sinh x\bigr) ^{2}}{2A_{t}}
 \right\} 
 \Biggr] .
\end{align*}
Differentiating the last expression with respect to $z$ entails that 
the former measure in \eqref{;twofm} admits the density $g_{1}$. 
Therefore $g_{1}(z)=g_{2}(z)$ for a.e.\ $z$, and hence for all 
$z$ by the continuity of $g_{1}$ and $g_{2}$. This proves the lemma.
\end{proof}

\begin{rem}\label{;rcont2}
As for the continuity of $g_{2}$ mentioned above, notice that, 
in view of formula \eqref{;fact}, the expectation in the definition of 
$g_{2}$ may be expressed as 
\begin{align*}
 \int _{0}^{\infty }du\,\frac{\cosh (x+z)}{\sqrt{2\pi u^{3}}}f(u)
 \ex \!\left[ 
 F(b^{z}_{s},s\le t)\exp \left\{ 
 -\frac{\bigl( \cosh (x+z)-u\cosh x/Z_{t}(b^{z})\bigr) ^{2}}{2u}
 \right\} 
 \right] ,
\end{align*}
by the independence of $b^{z}$ and $\hB $, and that 
$b^{z}\eqd \{ b^{0}_{s}+(z/t)s\} _{0\le s\le t}$.
\end{rem}

\begin{proof}[Proof of \tref{;tmain3}]
Since $x$ is arbitrary, we may take $x=-z$ in \lref{;lpinned}, 
which entails the relation 
\begin{align*}
 &\ex \!\left[ 
 F\bigl( 
 \tr _{B_{t}-z}(B)(s),s\le t
 \bigr) f(A_{t})\frac{1}{\sqrt{2\pi A_{t}}}
 \exp \left( 
 -\frac{e^{2B_{t}}\sinh ^{2}z}{2A_{t}}
 \right) 
 \right] \\
 &=\ex \!\left[ 
 F(b^{z}_{s},s\le t)f\bigl( 
 \tau _{1}(\hB ^{(\cosh z/Z_{t}(b^{z}))})
 \bigr) 
 \right] \frac{1}{\sqrt{2\pi t}}\exp \left( 
 -\frac{z^{2}}{2t}
 \right) .
\end{align*}
Multiplying both sides by $\cosh z$, we integrate both sides 
with respect to $z$ over $\R $. Then the left-hand side turns into 
\begin{align*}
 &\ex \!\left[ 
 F\bigl( 
 \tr _{B_{t}-\argsh \beta (e^{-2B_{t}}A_{t})}(B)(s),s\le t
 \bigr) f(A_{t})e^{-B_{t}}
 \right] \\
 &=e^{t/2}\ex \!\left[ 
 F\Bigl( 
 \tr _{
 \db{-1}_{t}-\argsh \beta \bigl( e^{-2\db{-1}_{t}}\!\!\da{-1}_{t}\bigr) 
 }(\db{-1})(s),s\le t\Bigr) f\bigl( \da{-1}_{t}\bigr) 
 \right] ,
\end{align*}
with $\beta $ independent of $B$, where the second line is due to 
the Cameron--Martin relation. On the other hand, the right-hand 
side turns into 
\begin{align*}
 e^{t/2}\ex \!\left[ 
 F\bigl( \db{\ve }_{s},s\le t\bigr) f\bigl( 
 \tau _{1}(\hB ^{(\cosh \db{\ve }_{t}/\dz{\ve }_{t})})
 \bigr) 
 \right] 
\end{align*}
by the Cameron--Martin relation. Here $B$, $\hB $ and are 
$\ve $ independent. Since the last two expressions agree for any 
$F$ and $f$, we obtain the identity in law 
\begin{equation}\label{;e1ptmain3}
\begin{split}
 &\left( 
 \left\{ 
 \tr _{\db{-1}_{t}-\argsh \beta \bigl( e^{-2\db{-1}_{t}}\!\!\da{-1}_{t}\bigr) }
 (\db{-1})(s)\right\} _{0\le s\le t},\,\da{-1}_{t}
 \right) \\
 &\eqd 
 \left( 
 \left\{ \db{\ve }_{s}\right\} _{0\le s\le t},
 \,\tau _{1}(\hB ^{(\cosh \db{\ve }_{t}/\dz{\ve }_{t})})
 \right) .
\end{split}
\end{equation}
By observing that  
\begin{align*}
 \left( 
 \{ \db{-1}_{s}\} _{0\le s\le t},\,e^{-2\db{-1}_{t}}\!\!\da{-1}_{t},
 \,\da{-1}_{t}
 \right) 
 \eqd 
 \left( 
 \left\{ R(\db{1})(s)\right\} _{0\le s\le t},\,\da{1}_{t},
 \,e^{-2\db{1}_{t}}\!\da{1}_{t}
 \right) 
\end{align*}
due to \eqref{;itrev}, 
the left-hand side of \eqref{;e1ptmain3} is identical in law with 
\begin{align*}
 &\left( 
 \left\{ 
 \tr _{-\db{1}_{t}-\argsh \beta (\da{1}_{t})}
 \bigl( R(\db{1})\bigr) (s)\right\} _{0\le s\le t},
 \,e^{-2\db{1}_{t}}\!\da{1}_{t}
 \right) \\
 &=\left( 
 \left\{ 
 R\bigl( \tr _{\db{1}_{t}+\argsh \beta (\da{1}_{t})}
 (\db{1})\bigr) (s)\right\} _{0\le s\le t},
 \,e^{-2\db{1}_{t}}\!\da{1}_{t}
 \right) ,
\end{align*}
where we used \pref{;pttrans}\thetag{v} for the second line. 
Therefore, by observing that $(R\circ R)(\phi )=\phi -\phi _{0}$ for every 
$\phi \in C([0,t];\R )$, we have, from 
\eqref{;e1ptmain3}, 
\begin{equation}\label{;e2ptmain3}
\begin{split}
 &\left( 
 \left\{ \tr _{\db{1}_{t}+\argsh \beta (\da{1}_{t})}
 (\db{1})(s)\right\} _{0\le s\le t},\,e^{-2\db{1}_{t}}\!\da{1}_{t}
 \right) \\
 &\eqd \left( 
 \left\{ 
 R(\db{\ve })(s)
 \right\} _{0\le s\le t},
 \,\tau _{1}(\hB ^{(\cosh \db{\ve }_{t}/\dz{\ve }_{t})})
 \right) .
\end{split}
\end{equation}
Note that $\db{\ve }_{t}$ and $\dz{\ve }_{t}$ in the second coordinate 
on the right-hand side may be expressed respectively 
as $-R(\db{\ve })(t)$ and $Z_{t}\bigl( R(\db{\ve })\bigr) $ 
by \lref{;ltrev}. Moreover, 
\begin{align*}
 \left\{ 
 R(\db{\ve })(s)
 \right\} _{0\le s\le t}
 \eqd \{ \db{\ve }_{s}\} _{0\le s\le t}
\end{align*}
because of \eqref{;itrev} and the independence of $B$ and $\ve $. 
Consequently, the right-hand side of \eqref{;e2ptmain3} is identical 
in law with that of the claimed identity \eqref{;etmain3}. As for 
the left-hand side of \eqref{;e2ptmain3}, we use the symmetry 
$\beta \eqd -\beta $ and the independence of $B$ and $\beta $ 
to see that it is identical in law with the left-hand side of 
\eqref{;etmain3}. Therefore identity \eqref{;etmain3} is proven. 
The latter identity in the theorem is shown by evaluating the 
first coordinates on both sides of \eqref{;etmain3} at $s=t$ and 
noting property \thetag{iv} in \pref{;pttrans}; in fact, we have 
\begin{align*}
 \left( 
 \beta \bigl( \da{1}_{t}\bigr) ,\,e^{-2\db{1}_{t}}\!\da{1}_{t},\,
 \{ \dz{1}_{s}\} _{0\le s\le t}
 \right) \eqd 
 \left( 
 \sinh \db{\ve }_{t},\,\tau _{1}(\hB ^{(\cosh \db{\ve }_{t}/\dz{\ve }_{t})}),
 \,\{ \dz{\ve }_{s}\} _{0\le s\le t}
 \right) ,
\end{align*} 
the identity between the third coordinates of which is consistent with 
\eqref{;oppz}. 
The proof of \tref{;tmain3} completes.
\end{proof}

\section{Concluding remarks}\label{;scr}

We conclude this paper with the following two remarks.

\begin{enumerate}[1.]{}
\item \tsref{;tmain1} and \ref{;tmain1d}, as well as \tref{;tmain2}, 
may be extended to the case where the deterministic time $t$ therein 
is replaced by any $\{ \Z _{s}\} $-stopping time $\tau $ satisfying 
$\pr (0<\tau <\infty )=1$. This is due to the fact that, as stated in 
\cite[Proposition~1.7]{myPI},  the assertion of \pref{;pmy} holds true 
with the above replacement because of the diffusion property 
of $\{ Z_{s}\} _{s\ge 0}$, as well as to the invariance of the path 
transformation $Z$ under the composition with $\tr _{z}$ as 
in \pref{;pttrans}\thetag{iv}, where we also consider a slight 
modification of the definition \eqref{;ttrans} of $\tr _{z}$ 
in accordance with replacement of $t$ by $\tau $.

\item \thetag{1} The path transformations $\tr _{z}$ provide an example 
of anticipative transformations studied by a number of researchers 
in the framework of the Malliavin calculus; see \cite{ram}, \cite{kus} 
and \cite{yano}, to name a few. Of concern in their studies is a 
Girsanov-type formula for those transformations, in which the 
density with respect to the underlying Wiener measure involves 
a Carleman--Fredholm determinant. Although we have not pursued it 
in this paper, we expect that \tref{;tmain2} will furnish an 
example for which the associated Carleman--Fredholm determinant 
is explicitly calculated. 

\thetag{2} Instead of going into details, here we content ourselves 
with showing a link between \tref{;tmain2} and the Malliavin calculus. 
For simplicity, let $F$ be a 
cylindrical functional on $C([0,t];\R )$ as of the form \eqref{;cyl}, 
with $f$ a bounded smooth function with bounded first derivatives 
which we denote by $\partial _{i}f,\,i=1,\ldots ,n$. 
Observe that 
\begin{align*}
 \frac{d}{dz}F\bigl( 
 \tr _{z}(B)(s),s\le t
 \bigr) 
 =-\left\langle \!
 (DF)(\tr _{z}(B)),
 \frac{\int _{0}^{\cdot }e^{2\tr _{z}(B)(s)}\,ds}{A_{t}(\tr _{z}(B))}
 \right\rangle _{H},\quad z\in \R ,
\end{align*}
where $\langle \cdot ,\cdot \rangle _{H}$ stands for the inner 
product on the Cameron--Martin subspace $H$ of $C([0,t];\R )$, 
and $DF$ denotes the Malliavin derivative of $F$, namely, 
\begin{align*}
 DF(\phi )=\sum _{i=1}^{n}
 \partial _{i}f(\phi _{t_{1}},\ldots ,\phi _{t_{n}})
 \int _{0}^{\cdot }\ind _{[0,t_{i}]}(s)\,ds,
 \quad \phi \in C([0,t];\R ).
\end{align*}
Then, taking the derivative at $z=0$ on both sides of \eqref{;t21} yields 
\begin{align}\label{;ibp}
 \ex \!\left[ 
 \biggl\langle DF(B),\frac{\int _{0}^{\cdot }e^{2B_{s}}\,ds}{A_{t}}
 \biggr\rangle _{H}
 \right] 
 =\ex \!\left[ 
 \frac{\sinh B_{t}}{Z_{t}}F(B_{s},s\le t)
 \right] .
\end{align}
We denote by $\delta $ the Skorokhod integral, namely the adjoint of 
the operator $D$; see \cite[Definition~1.3.1]{nua}. 
Equation \eqref{;ibp} may be explained in terms of $\delta $. 
Indeed, by the definition of $\delta $, the left-hand side of 
\eqref{;ibp} is rewritten as 
\begin{align}\label{;ibpr}
 \ex \!\left[ 
 F(B_{s},s\le t)\delta \left( 
 \frac{\int _{0}^{\cdot }e^{2B_{s}}\,ds}{A_{t}}
 \right) 
 \right] ,
\end{align}
in which, by \cite[Proposition~1.3.3]{nua}, the Skorokhod integral 
is expanded into 
\begin{align}\label{;div}
 \frac{1}{A_{t}}\delta \left( \int _{0}^{\cdot }e^{2B_{s}}\,ds\right) 
 -\left\langle D\left( \frac{1}{A_{t}}\right) ,
 \int _{0}^{\cdot }e^{2B_{s}}\,ds
 \right\rangle _{H}.
\end{align}
Notice that 
\begin{align*}
 \delta \left( \int _{0}^{\cdot }e^{2B_{s}}\,ds\right) 
 &=\int _{0}^{t}e^{2B_{s}}\,dB_{s}\\
 &=\frac{1}{2}\left( e^{2B_{t}}-1\right) -A_{t}, 
\end{align*}
where the second line is due to It\^o's formula. On the other hand, 
since the directional derivative of the functional $1/A_{t}$ along 
each $h\in H$ is 
\begin{align*}
 -\frac{2}{(A_{t})^{2}}\int _{0}^{t}h_{s}e^{2B_{s}}\,ds,
\end{align*}
the $H$-inner product in \eqref{;div} is equal to $-1$.
These amount to 
\begin{align*}
 \delta \left( 
 \frac{\int _{0}^{\cdot }e^{2B_{s}}\,ds}{A_{t}}
 \right) 
 &=\frac{e^{2B_{t}}-1}{2A_{t}}\\
 &=\frac{\sinh B_{t}}{Z_{t}}
\end{align*}
by the definition of $Z_{t}$, which, together with \eqref{;ibpr}, 
verifies \eqref{;ibp}. Similarly, differentiation of both sides of 
\eqref{;t21} at any fixed $z$, with relation \eqref{;t21} applied to, 
also yields 
\begin{align*}
 &\ex \!\left[ 
 \exp \left\{ 
 \frac{\cosh B_{t}-\cosh (z+B_{t})}{Z_{t}}
 \right\} 
 \biggl\langle DF(B),\frac{\int _{0}^{\cdot }e^{2B_{s}}\,ds}{A_{t}}
 \biggr\rangle _{H}
 \right] \\
 &=\ex \!\left[ 
 \frac{\sinh (z+B_{t})}{Z_{t}}\exp \left\{ 
 \frac{\cosh B_{t}-\cosh (z+B_{t})}{Z_{t}}
 \right\} F(B_{s},s\le t)
 \right] .
\end{align*}
This equality may be verified in the same manner 
as above; indeed, a direct computation shows that 
\begin{align*}
 &\delta \left( 
 \exp \left\{ 
 \frac{\cosh B_{t}-\cosh (z+B_{t})}{Z_{t}}
 \right\} \frac{\int _{0}^{\cdot }e^{2B_{s}}\,ds}{A_{t}}
 \right) \\
 &=\frac{\sinh (z+B_{t})}{Z_{t}}\exp \left\{ 
 \frac{\cosh B_{t}-\cosh (z+B_{t})}{Z_{t}}
 \right\} .
\end{align*}
\end{enumerate}


\end{document}